\documentclass[12pt]{amsart}
\usepackage{amsfonts,amssymb,amscd,amsmath,enumerate,verbatim,calc}
\usepackage[colorlinks,linkcolor=blue,pagebackref=true, citecolor=red]{hyperref}
\usepackage[all]{xy}
\usepackage{multicol}
\usepackage[mathscr]{euscript} 
\usepackage{mathrsfs}
\usepackage[left=2cm,right=2cm,top=2.5cm,bottom=2.5cm]{geometry}

\usepackage{graphicx}
\usepackage{color}
\usepackage{tikz}
\usepackage{tikz-cd}

\usepackage{vwcol}

\newcommand{\m}{\mathfrak{m} }

\newcommand{\f}{\mathcal{F}}
\newcommand{\R}{\mathcal{R}}

\newcommand{\Z}{\mathbb{Z} }
\newcommand{\N}{\mathbb{N} }

\newcommand{\QQ}{\mathbb{Q} }

\newcommand{\C}{\mathbb{C} }

\providecommand{\Q}{{\mathbb Q}}

\newcommand{\Ass}{\operatorname{Ass}}
\newcommand{\chars}{\operatorname{char}}

\newcommand{\grade}{\operatorname{grade}}

\newcommand{\gr}{\operatorname{gr}}

\newcommand{\sing}{\operatorname{Sing}}

\newcommand{\hgt}{\operatorname{ht}}

\newcommand{\vol}{\operatorname{vol}}

\newcommand{\sym}{\operatorname{Sym}}

\theoremstyle{plain}
\newtheorem{theorem}{Theorem}[section]

\newtheorem{proposition}[theorem]{Proposition}

\newtheorem*{algorithm*}{Algorithm}

\theoremstyle{definition}

\newtheorem{remark}[theorem]{Remark}
\newtheorem{remarks}[theorem]{Remarks}
\newtheorem{example}[theorem]{Example}
\newtheorem*{example*}{\it Example}

\theoremstyle{remark}
\newtheorem*{claim*}{\it Claim}

\newtheorem*{case*}{\it Case}
\newtheorem*{note*}{\it Note}

\title[Computing mixed multiplicities]{COMPUTING MIXED MULTIPLICITIES, MIXED VOLUMES AND SECTIONAL MILNOR NUMBERS}

\author{Kriti Goel}
\address{Department of Mathematics, University of Utah, Salt Lake City, UT 84112, USA}
\email{kritigoel.maths@gmail.com}

\author{Vivek Mukundan}
\address{Department of Mathematics, Indian Institute of Technology Delhi, New Delhi, 110016, India}
\email{vmukunda@iitd.ac.in}

\author{Sudeshna Roy}
\address{Department of Mathematics, Chennai Mathematical Institute, Siruseri, Kelambakkam 603103, India}
\email{sudeshnaroy.11@gmail.com}

\author{J. K. Verma}
\address{Department of Mathematics, Indian Institute of Technology Bombay, Mumbai, 400076, India}
\email{jkv@math.iitb.ac.in}

\begin{document}

\begin{abstract}
This is an expository version of our paper \cite{GoeMukRoyVer}. Our aim is to present recent Macaulay2 algorithms for computation of mixed multiplicities of ideals in a Noetherian ring which is either local or a standard graded algebra over a ﬁeld. These algorithms are based on computation of the equations of multi-Rees algebras of ideals that generalises a result of  Cox, Lin and Sosa \cite{CLS}. Using these equations we propose efficient algorithms for computation of mixed volumes of convex lattice  polytopes and sectional Milnor numbers of hypersurfaces with an isolated singularity. 
\end{abstract}

\maketitle

\section{Introduction}
In this expository article we describe the Macaulay2 package \texttt{MixedMultiplicity} which computes the mixed multiplicities of ideals having positive grade in a Noetherian ring, mixed volume of a collection of convex lattice polytopes and sectional Milnor numbers of hypersurfaces with an isolated singularity. One of the main steps of the algorithms is the computation of the defining equations of the multi-Rees algebra 
\[\R(I_1,\ldots,I_s)=R[I_1t_1, \ldots, I_st_s]=\bigoplus\limits_{a_1, \ldots, a_s \geq 0} I_1^{a_1} \cdots I_s ^{a_s}t_1^{a_1} \cdots t_s^{a_s} \subseteq R[t_1, \ldots, t_s]\] 
of ideals $I_1,\ldots,I_s$ in a Noetherian ring $R$. When $R$ is a polynomial ring of finite Krull dimension over any field $k$ and $I_1,\ldots, I_s$ are monomial ideals in $R$, D. Cox, K.-N. Lin, and G. Sosa give an explicit formula for the defining equations of $\R(I_1,\ldots,I_s)$ in \cite{CLS}. In this article, we obtain an analogue of their result for any set of ideals $I_1, \ldots, I_s$ in a Noetherian ring $R$ such that each $I_i$ has positive grade. The latter condition is always satisfied when $R$ is a domain or for any ideal of positive height in a reduced ring or in a Cohen-Macaulay ring. Several authors have proposed algorithms to determine the defining equations of the multi-Rees algebra. For example,  J. Ribbe \cite[Proposition 3.1, Proposition 3.4]{Ribbe}, K.-N. Lin and C. Polini \cite[Theorem 2.4]{Lin-Polini}, G. Sosa \cite[Lemma 2.1]{Sosa},  and  B. Jabarnejad \cite[Theorem 1]{Jabarnejad}. They  assumed $I_i=\m^{a_i}$ or $I_i=I^{a_i}$, where $I$ is an ideal of linear type, or $I_i$'s are monomial ideals. We show the following.

\begin{theorem}
	Let $R$ be a Noetherian ring and $I_1,\ldots, I_s \subseteq R$ be ideals of positive grade. For each $i$, consider some generating set $\{f_{ij} \mid j=1, \ldots, n_i\}$ of $I_i$ which contains at least one nonzerodivisor $f_{ij_i}$. We set $h=\prod_{i=1}^s f_{ij_i}$. Consider the indeterminates $\underline{Y} = \{Y_{ij} \mid i=1,\ldots,s, j=1,\ldots,n_i \}$ and $\underline{T}=(T_1,\ldots,T_s).$ Define an $R$-algebra homomorphism
	$\varphi: S=R[\underline{Y}] \to \R(I_1,\ldots,I_s)\subseteq R[\underline{T}]$ by
	$Y_{ij} \mapsto f_{ij}T_i$ for all $i=1,\ldots,s, \text{ and } j=1,\ldots,n_i.$ Then
	\[ \Gamma := \langle Y_{ij} f_{ij'} - Y_{ij'} f_{ij} \mid i=1,\ldots,s \text{ and } j,j' \in \{1,\ldots,n_i\} \mbox{ with } j \neq j' \rangle : h^\infty \subseteq R[\underline{Y}]. \]	
	is the defining ideal of $\mathcal{R}(I_1,\ldots,I_s)$, that is, $\Gamma=\ker \varphi$. 
\end{theorem}

Using this result, we write a function \texttt{multiReesIdeal} to compute the defining ideal of a multi-Rees algebra in Macaulay2. 
When the ring is not a domain, the algorithm follows an analogue of the \texttt{reesIdeal} function, albeit for the multiple ideal setting.

The algorithm for computing defining equations of $\R(I_1, \ldots, I_s)$ helps us construct algorithms to compute \emph{mixed multiplicities}, \emph{mixed volume} and \emph{sectional Milnor numbers} in the general setting. To justify importance of such an algorithm, we now recall these invariants and their importance.

\subsection{\bf Mixed multiplicities of ideals}
Let $I_1,\ldots, I_r$ be a sequence of ideals in a local ring $(R,\m)$ {\rm(}or a standard graded algebra over a field and $\m$ is the maximal graded ideal{\rm)} and let $I_0$ be an $\m$-primary ideal. In \cite{Verma-katz}, the authors prove that $\ell\left(I_0^{u_0}I_1^{u_1}\cdots I_r^{u_r}/I_0^{u_0+1}I_1^{u_1}\cdots I_r^{u_r}\right)$ is a polynomial $P(\underline{u})$, for $u_i$ large, where $\underline{u}=(u_0, \ldots, u_r)$. Write this polynomial in the form 
\[P(\underline{u})=\sum_{\substack{\alpha \in \N^{r+1}\\ |\alpha|=t}} \frac{1}{\alpha!}e_{\alpha}(I_0 \mid I_1,\ldots,I_r) u^{\alpha}+\text{ lower degree terms, }\] 
where $t = \deg P(\underline{u}), \alpha=(\alpha_0, \ldots, \alpha_r) \in \N^{r+1}, \alpha! = \prod_{i=0}^{r} \alpha_i!$ and $|\alpha| = \sum_{i=0}^{r} \alpha_i.$ If $|\alpha|=t$, then $e_{\alpha}(I_0 \mid I_1,\ldots,I_r)$ are called the {\it mixed multiplicities} of the ideals $I_0, I_1,\ldots,I_r$. In \cite{Trung-Verma}, the authors prove the following result.

\begin{theorem}[{\rm{\cite[Theorem 1.2]{Trung-Verma}}}]\label{deg-mixedmul}
	Assume that $d = \dim R/(0 :I^{\infty})\geq 1$, where $I=I_1\cdots I_r$. Then $\deg P(\underline{u})=d-1$.
\end{theorem}

We now briefly describe several applications of mixed multiplicities of ideals. Suppose that $(R,\m)$ is a local ring of positive dimension $d$ and $I_0 \ldots, I_r \subseteq R$ are ideals of positive height.

\noindent
(i) Let $t_0, \ldots, t_r$ be indeterminates. Consider the multi-form rings 
\begin{equation}\label{gr}
\gr_R(I_0, \ldots, I_r; I_i)= \R(I_0, \ldots, I_r)/(I_i) \quad \text{for } i =0, 1, \ldots, r.
\end{equation}
Note that if $I_0$ is an $\m$-primary ideal and $I_1, \ldots, I_r$ are ideals of positive height, then \[\ell\left(\frac{I_0^{u_0}I_1^{u_1} \cdots I_r^{u_r}}{I_0^{u_0+1}I_1^{u_1} \cdots I_r^{u_r}}\right)=\ell([\gr_R(I_0, \ldots, I_r; I_0)]_{\underline{u}})\] is a polynomial of total degree $d-1$ for $u_i \gg 0$ and $i=0, 1, \ldots, r$ as we have observed above. Let us rewrite the terms in total degree $d-1$ in the form 
\[\sum_{\substack{\alpha \in \N^{r+1},\; |\alpha|=d-1}} \frac{1}{\alpha!}e(I_0^{[\alpha_0+1]} |I_1^{[\alpha_1]}| \cdots |I_r^{[\alpha_r]})u_0^{\alpha_0}\cdots u_r^{\alpha_r}.\]
For this case, M. Herrmann and his coauthors \cite{HERRMANN} expressed the multiplicity of $\gr_R(I_0, \ldots, I_r; I_0)$ by means of mixed multiplicities in the following way: 
\[e(\gr_R(I_0, \ldots, I_r; I_0))=\sum_{\substack{\alpha \in \N^{r+1}\\ |\alpha|=d-1}}e(I_0^{[\alpha_0+1]} |I_1^{[\alpha_1]}| \cdots |I_r^{[\alpha_r]}).\] 
(ii) In \cite{VERMA1992}, J. K. Verma proved that 
\[e(\R(I_1 \ldots, I_r))=e(\gr_R(\m,I_1 \ldots, I_r; \m))=\sum_{\substack{\alpha \in \N^{r+1}, \;  |\alpha|=d-1}}e(\m^{[\alpha_0+1]} |I_1^{[\alpha_1]}| \cdots |I_r^{[\alpha_r]}).\]
(iii) In \cite{D'Cruz2003}, C. D'Cruz showed similar result for the multi-graded extended Rees algebra $T:=\mathcal{B}(I_1, \ldots, I_r)$ $= R[I_1t_1, \ldots, I_rt_r, t_1^{-1}, \ldots, t_r^{-1}]$. She showed that 
\[e(T_N)=\frac{1}{2^{d}}\left[\sum\limits_{j=0}^r \sum\limits_{\substack{\beta+\beta_1+\cdots+\beta_j=d-1 \\ 1 \leq i_1< \cdots<i_j \leq r}}2^{\beta_1+\cdots+\beta_j} e(L^{[\beta+1]} | I_{i_1}^{[\beta_1]}| \cdots |I_{i_j}^{[\beta_j]})\right],\] 
where $N:=\mathcal{N}(I_1, \ldots, I_r)=(I_1t_1, \ldots, I_rt_r, \m, t_1^{-1}, \ldots, t_r^{-1})$ and $L=\m^2+I_1+\cdots+I_r$.

In view of the above, we can easily compute multiplicities of Rees algebras, extended Rees algebras and certain form rings, once mixed multiplicities are known. 

\subsection{\bf Mixed volumes of lattice polytopes.}\label{MV}
We first recall the definition of mixed volumes. Let $P,Q$ be two polytopes (not necessarily distinct) in $\mathbb{R}^n.$ Their Minkowski sum is defined as the polytope
\[P + Q := \{a + b \mid a \in P, b \in Q \}.\]
Let $Q_1,\ldots,Q_n$ be an arbitrary collection of lattice polytopes in $\mathbb{R}^n$ and  $\lambda_1, \ldots, \lambda_n \in \mathbb{R}_{+}$. H. Minkowski proved that the function $\vol_n(\lambda_1Q_1 + \cdots + \lambda_n Q_n)$ is a homogeneous polynomial of degree
$n$ in $\lambda_1, \ldots, \lambda_n$, see \cite[Theorem 4.9, Chapter 7]{CLO2}. The coefficient of $\lambda_1 \cdots \lambda_n$ is called {\it the mixed volume} of $Q_1,\ldots,Q_n$ and denoted by $MV_n(Q_1,\ldots, Q_n).$ We set  $[n]:=\{1, \dots, n\}.$ Then
\[ MV_n(Q_1,\ldots,Q_n) :=\sum_{\emptyset \neq J \subseteq [n]} (-1)^{n-|J|} V_n\left(\sum_{j \in J} Q_j\right),\]
see \cite[Theorem 4.12 d., Chapter 7]{CLO2}. Here $V_n$ denotes the $n$-dimensional Euclidean volume. 

In \cite{Trung-Verma}, the authors explore the relationship between mixed multiplicities of multigraded rings and mixed volumes. 

\begin{proposition}[{\rm {\cite[Corollary 2.5]{Trung-Verma}}}]\label{vol-mul}
	Let $Q_1,\ldots,Q_n$ be lattice polytopes in $\mathbb{R}^n$. Suppose that $R = k[x_0, x_1,\ldots, x_n]$ is a polynomial ring and $\m$ is its graded maximal ideal. Let $M_i$ be any set of monomials of the same degree in $R$ such that $Q_i$ is the convex hull of the lattice points of their dehomogenized monomials in $k[x_1,\ldots,x_n]$. Let $I_j$ be the ideal of $R$ generated by the monomials of $M_j$. Then\[MV_n(Q_1,\ldots,Q_n) = e_{(0,1,...,1)}(\m|I_1,\ldots, I_n).\]	
\end{proposition}	

Mixed volumes of lattice polytopes have diverse applications. We describe a few of them. 

\vspace{0.2cm}
(1) In 1975, D. N. Bernstein proved that the mixed volume of the  Newton polytopes of Laurent polynomials $f_1, \ldots, f_n \in k[x_1^{\pm 1}, \ldots, x_n^{\pm 1}]$ is a sharp upper bound for the number of isolated common zeros in the torus $(k^*)^n$ where $k$ is an algebraically closed field. Furthermore, this bound is attained for a generic choice of coefficients in these $n$ polynomials if  $\chars k=0$.

(2) In 1993, W. Fulton proved that the mixed volume of the Newton polytopes of Laurent polynomials $f_1, \ldots, f_n \in k[x_1^{\pm 1}, \ldots, x_n^{\pm 1}]$ is an upper bound of $\sum_{\alpha}i(\alpha)$, where $\alpha \in (\overline{k}^*)^n$ is any isolated common zero and $i(\alpha)$ denotes the intersection multiplicity at $\alpha.$

(3) Let $P$ be a lattice polytope of dimension $n$ and $|P \cap \Z^n|$ denote the number of lattice points. In 1962, E. Ehrhart proved that the function $L(t \cdot P)=|tP \cap \Z^n|$, for all non-negative integers $t$, is a polynomial of degree $n$ denoted by 
$E_P(t)\in \QQ[t]$. The polynomial $E_P(t)$ is called the \emph{Ehrhart polynomial} of $P$. For lattice polytopes $P_1, \ldots, P_k \subseteq \mathbb{R}^n$ in the integer lattice $\Z^n$, we have the following {\it mixed Ehrhart polynomial} in one variable $t$: \[ME_{P_1, \ldots, P_k}(t)=\sum_{\emptyset \neq J \subseteq [k]} (-1)^{k-|J|} L(t \cdot \sum_{j \in J} P_j) \in \QQ[t].\]
If $\dim\left(\sum_{i=1}^k P_i\right)=n$, then by \cite[Corollary 2.5]{HaaJuhMarRamTho}, 
the leading coefficient of $ME_{P_1, \ldots, P_k}(t)$ is 
\[\sum \limits_{\substack{\sum_{i=1}^k s_i=n, ~s_i \geq 1}}\binom{n}{s_1,\ldots, s_k}MV_n\left(P_1[s_1], \ldots, P_k[s_k]\right),\]
where $\binom{n}{s_1,\ldots, s_k}:=\frac{n!}{s_1!\cdots s_k!}$ and $MV_n\left(P_1[s_1], \ldots, P_k[s_k]\right)$ means that the polytope $P_j$ is taken $s_j$ times. In particular, when $k=n$, 
\[ME_{P_1, \ldots, P_n}(t)= n! MV_n(P_1,\ldots,P_n) t^n.\]
The coefficient of $t^{n-1}$ in $ME_{P_1, \ldots, P_k}(t)$ is also described in \cite[Corollary 2.6]{HaaJuhMarRamTho}.

\subsection{\bf  Sectional Milnor numbers.} 
Suppose that the origin is an isolated singular point of a complex analytic hypersurface $H=V(f)\subset \mathbb{C}^{n+1}.$ Let $f_{z_i}$ denote the partial derivative of $f$ with respect to $z_i.$ Set
\[\mu=\dim_{\mathbb{C}}\frac{\mathbb{C}\{z_0, z_1, \dots, z_n\}}{(f_{z_0}, f_{z_1},\dots, f_{z_n})}.\]
The number $\mu$ is called the {\it Milnor number} of the hypersurface $H$ at the origin. In his Carg\`ese paper \cite{teissier1973}, B. Teissier refined the notion of Milnor number by replacing it with a sequence of Milnor numbers of intersections with general linear subspaces. Let $(X, 0)$ be a germ of a hypersurface in  $\mathbb{C}^{n+1}$ with an isolated singularity. The Milnor number of $X\cap E$, where $E$ is a general linear subspace of dimension $i$ passing through the origin, is called the {\it $i^{th}$-sectional Milnor number} of $X$ at the origin. It is denoted by $\mu^{(i)}(X, 0).$ These are collected together in  the sequence
\[\mu^*(X, 0) = (\mu^{(n+1)}(X, 0), \mu^{(n)} (X, 0), \ldots , \mu^{(0)}(X, 0)).\]

Let $R=\mathbb{C}[x_0,x_1,\ldots,x_n]$ be a polynomial ring in $n+1$ variables, $\mathfrak{m}$ be the maximal graded ideal and $f \in R$ be any polynomial having an isolated singularity at the origin. Recall that a point $a \in V(f)$ is said to be \emph{singular} if the rank of the $1 \times n$ matrix $J(f)|_{a}$ is zero, where $J(f)=(f_{x_0},f_{x_1},\ldots,f_{x_n})$ denotes the \emph{Jacobian matrix}. A polynomial $f$ is said to have an isolated singularity at a point $a$ if $\{a\}$ is a (connected) component of the variety $\sing(f)=\{a \in V(f) \mid J(f)|_{a}=0\}$. We use the same notation $J(f)$ to denote the \emph{Jacobian ideal} which is generated by the entries of the Jacobian matrix associated to $f$. Teissier proved that the $i^{th}$-mixed multiplicity, denoted by  $e_i(\m|J(f))$ or $e(\m^{[n+1-i]}, J(f)^{[i]})$ is equal to the $i^{th}$-sectional Milnor number of the singularity. Note that the the $(n+1)^{th}$ sectional Milnor number of $f$ is the Minor number of $f,$ namely, $\dim_CR/J(f).$

Let $f \in \C[x_0, x_1, \ldots, x_n]$ be a homogeneous polynomial in positive degree. Set $D(f)= \mathbb{P}^n \backslash V(f)$. In \cite[p. 6]{June12}, J. Huh showed that {\it the Euler Characteristic} of $D(f)$ is given by 
\[\chi(D(f))=\sum\limits_{i=0}^n (-1)^i e_i(\m \mid J(f)).\] 

We now describe the  contents of the paper. In section 2, we generalize the result of Cox, Lin and Sosa to find the defining ideal of multi-Rees algebra of any collection of ideals having positive grade in a local ring or in a standard graded algebra over a field. For any ideal $J$ in a Noetherian ring $B$, we define the {\it grade of $J$} to be the common length of the maximal $B$-sequences in $J$ and denote it by $\grade(J)$. In view of our result, we write a Macaulay2 algorithm to compute the defining ideal. The computation time of these algorithms is compared with the computation time of the Macaulay2 function \texttt{reesIdeal} in \cite{GoeMukRoyVer}.

In section 3, we present a Macaulay2 algorithm which calculates mixed multiplicities of ideals in a polynomial ring. We start with generalizing a result of D. Katz, S. Mandal and J. K. Verma, to give a precise formula for the Hilbert polynomial of the quotient of a multi-graded algebra over an Artinian local ring. This  result helps to calculate the mixed multiplicities of a set of ideals $I_0, I_1, \ldots, I_r$, where $I_0$ is primary to the maximal ideal and $I_0, I_1, \ldots, I_r$ are ideals in a local ring or in a standard graded algebra over a field.  The algorithm computes their mixed multiplicities, where $\grade(I_j)>0$ for all $j.$ 

Section 4 contains an in depth explanation of the codes which are used to write the function \texttt{mixedMultiplicity} in the package \href{https://faculty.math.illinois.edu/Macaulay2/doc/Macaulay2-1.19.1/share/doc/Macaulay2/MixedMultiplicity/html/index.html}{\texttt{MixedMultiplicity}}. This is meant for the new users of Macaulay2.

In section 5, we give an algorithm which computes mixed volume of a collection of lattice  polytopes in $\mathbb{R}^n.$ We start with an algorithm which outputs the homogeneous ideal corresponding to the vertices of a lattice  polytope. Given a collection of lattice polytopes in $\mathbb{R}^n$, N. V. Trung and the last author proved that their mixed volume is equal to a mixed multiplicity of a set of homogeneous ideals. The function \texttt{mMixedVolume} is written based on their results to calculate the mixed volume. 

In the last section, we give an algorithm to compute the sectional Milnor numbers. We use Teissier's observation of identifying the sectional Milnor numbers with mixed multiplicities to achieve this task. Let $f$ be a polynomial in the ring $k[x_1,\ldots,x_n],$ where $\text{char} (k) = 0$ and $J(f)$ is $\m$-primary. With an input of the polynomial $f$, the function \texttt{secMilnorNumbers} outputs $e_0(\m \mid J(f)), \ldots, e_{n}(\m \mid J(f)).$ Teissier \cite{teissier1973}   conjectured that invariance of the Milnor number of an isolated singularity in a family of hypersurfaces implies invariance of the sectional Milnor numbers. The conjecture was disproved by Jo\"el Brian\c{c}on and Jean-Paul Speder. We verify their example using our algorithm and also by explicitly calculating the mixed multiplicities.

The algorithms written for the computation of defining ideal, mixed multiplicities, mixed volume and sectional Milnor numbers are used in the Macaulay2 (\cite{M2}) package \texttt{MixedMultiplicity}. Macaulay2 is a software system devoted to support research in algebraic geometry and commutative algebra.

\section{Defining equations of multi-Rees algebras of ideals}
An explicit formula for the defining ideal of the multi-Rees algebra of a finite collection of monomial ideals in a polynomial ring was given by D. Cox, K.-N. Lin, and G. Sosa in \cite{CLS}. In this section, we generalize their result to find the defining ideal of the multi-Rees algebra of a collection of ideals with positive grade in a Noetherian ring. We use this result to write a Macaulay2 algorithm to compute the defining ideal when the base ring is a domain. We provide another algorithm for the non-domain case. 

Let $R$ be a Noetherian ring and $I_1,\ldots, I_s \subseteq R$ be ideals. Suppose that $I_i = \langle f_{ij} \mid j=1,\ldots,n_i \rangle$ for all $i = 1,\dots,s$. Let $\R(I_1,\ldots,I_s)$ be the multi-Rees algebra of ideals $I_1,\ldots,I_s.$ Consider the set of indeterminates $\underline{Y} = \{Y_{ij} \mid i=1,\ldots,s, j=1,\ldots,n_i \}$ and $\underline{T}=(T_1,\ldots,T_s)$ over $R.$ Define an $R$-algebra homomorphism
$R[\underline{Y}] \overset{\varphi} \longrightarrow \R(I_1,\ldots,I_s) \subseteq R[\underline{T}]$
such that $\varphi(Y_{ij}) = f_{ij}T_i$, for all $i=1,\ldots,s$, $j=1,\ldots,n_i$ and $\varphi(r) = r$ for all $r \in R.$ Then $\mathcal{R}(I_1,\ldots,I_s) \simeq R[\underline{Y}] / \ker(\varphi).$ The ideal $\ker \varphi$ is called the defining ideal of $\R(I_1, \ldots, I_s)$. We give an explicit description of $\ker(\varphi).$ 

\begin{theorem} \label{TisRees}
Let $R$ be a Noetherian ring and $I_1,\ldots, I_s \subseteq R$ be ideals of positive grade. For each $i$, consider a generating set $\{f_{ij} \mid j=1, \ldots, n_i\}$ of $I_i$ which contains at least one nonzerodivisor $f_{ij_i}$. We set $h=\prod_{i=1}^s f_{ij_i}$ and set
\[ \Gamma = \langle Y_{ij} f_{ij'} - Y_{ij'} f_{ij} \mid i=1,\ldots,s \text{ and } j, j' \in \{1,\ldots,n_i\} \mbox{ with } j \neq j' \rangle : h^\infty \subseteq R[\underline{Y}]. \]	
Then $\Gamma \subseteq R[\underline{Y}]$ is the defining ideal of $\mathcal{R}(I_1,\ldots,I_s)$. 
\end{theorem}

\begin{proof}
Without loss of generality, we may assume that $j_i=1$ for all $i=1,\ldots,s$ and $h=\prod_{i=1}^s f_{i1}$. Consider the ring homomorphism $\phi: R \to R[f_{11}^{-1}, f_{21}^{-1}, \ldots, f_{s1}^{-1}] \cong R[h^{-1}]$, which induces a natural map $\widetilde{\phi}: R[\underline{Y}] \to R[h^{-1}][\underline{Y}]\cong R[\underline{Y}]_h$. Note that $\phi(I_i)=(1, f_{i2}/f_{i1}, \ldots, f_{in_i}/f_{i1})\subseteq R[h^{-1}]$ is the unit ideal for $i=1, \ldots, s$. Consider the ring homomorphism 
\[\theta: R[h^{-1}][\underline{Y}] \to \mathcal{R}(\phi(I_1),\ldots,\phi(I_s))\subseteq R[h^{-1}][\underline{T}]\] 
sending $Y_{i1}$ to $T_i$ and $Y_{ij}$ to $\frac{f_{ij}}{f_{i1}}T_i$ for $i=1, \ldots, s$ and $j=1, \ldots, n_i$. Then we show that the defining ideal $\ker \theta$ of $\mathcal{R}(\phi(I_1),\ldots,\phi(I_s))$ is $J:=J_1+\cdots+J_s$, where $J_i:=(Y_{i2}-\frac{f_{i2}}{f_{i1}}Y_{i1}, \ldots,  Y_{in_i}-\frac{f_{in_i}}{f_{i1}}Y_{i1})$. Clearly, $J \subseteq \ker \theta$. On the other hand,
\[\frac{R[h^{-1}][\underline{Y}]}{J} \cong R[h^{-1}][Y_{i1} \mid i=1, \ldots, s] \cong R[h^{-1}][\underline{T}].\]
This forces $\ker \theta =J$, as the last isomorphism is via $\theta$. We further claim that $\widetilde{\phi}^{-1}(J) = \Gamma.$ For all $j \neq j',$ and every $i$,
\[f_{ij}Y_{ij'}-f_{ij'}Y_{ij}=f_{ij}\left(Y_{ij'}-\frac{f_{ij'}}{f_{i1}}Y_{i1}\right)-f_{ij'}\left(Y_{ij}-\frac{f_{ij}}{f_{i1}}Y_{i1}\right)\] 
is in $J_i$. So $\Gamma \subseteq \widetilde{\phi}^{-1}(J)$. Now let $r \in \widetilde{\phi}^{-1}(J)$. Then $\widetilde{\phi}(r) \in J$, i.e.,
\[\frac{r}{1} =\sum_{i=1}^s\sum_{j=2}^{n_i} \frac{a_{ij}}{h^{m_{ij}}}\left(Y_{ij}-\frac{f_{ij}}{f_{i1}}Y_{i1}\right)\]
for some $a_{ij} \in R[\underline{Y}]$. Thus we have 
\[h^{m} r \in (f_{i1}Y_{ij}-f_{ij}Y_{i1}\mid 1 \leq i \leq s \text{ and } 1 \leq j \leq n_i)\subseteq (f_{ij}Y_{ij'}-f_{ij'}Y_{ij} \mid 1 \leq i \leq s \text{ and } 1 \leq j, j' \leq n_i) \]
for some $m \geq \max\{m_{ij}\mid 1 \leq i \leq s, 1 \leq j \leq n_i\}+1$. Therefore, $r \in \Gamma$ and hence the claim holds. Next, observe that $h^t$ is a nonzerodivisor on $R$ and hence on $R[\underline{Y}]/\mathcal{K}\cong \mathcal{R}(I_1, \ldots, I_s)$ for every $t \geq 1$. So $\mathcal{K}$ is a contracted ideal by \cite[Proposition 3.11 $\mbox{(iii)}$]{Atiyah-MacDonald}. Besides, $ \mathcal{R}(\phi(I_1),\ldots,\phi(I_s)) \cong \mathcal{R}(I_1, \ldots, I_s)_h$. Therefore, \[R[h^{-1}][\underline{Y}]/J \cong \mathcal{R}(I_1, \ldots, I_s)_h \cong \left(R[\underline{Y}]/\mathcal{K}\right)_h \cong  R[h^{-1}][\underline{Y}]/\mathcal{K}_h.\] 
As $J=\ker \theta=\mathcal{K}_h=\langle\widetilde{\phi}(\mathcal{K})\rangle$ and $\mathcal{K}$ is a contracted ideal so we get that $\mathcal{K}=\widetilde{\phi}^{-1}(J)=\Gamma$.
\end{proof}

When $R$ is a domain or when a list of nonzerodivisors (one each from the list of ideals with positive grades) is provided by the user, the function \texttt{multiReesIdeal} computes the defining ideal of the multi-Rees algebra using Theorem \ref{TisRees}.  

\noindent
\begin{algorithm*}[Version I: \texttt{multiReesIdeal}, set of ideals with positive grade]
Let $I_1,\ldots,I_s$ be ideals of a Noetherian ring $R$ with $\grade I_i>0$ for all $i$ and let $a_1, \ldots, a_s$ be a set of nonzerodivisors, where $a_i$ belongs to the generating set of $I_i$ for all $i.$ When $R$ is a domain, the function picks $a_i$ to be the first element in the generating set of $I_i$ for each $i$.\\
{\bf Input}: The list $W = \{\{I_1,\ldots,I_s\},\{a_1, \ldots, a_s\}\}$, or $W = \{I_1,\ldots,I_s\}$ if $R$ is a domain.
\begin{enumerate}[1.]
	\item Define a polynomial ring $S$ by attaching $m$ indeterminates to the ring $R$, where $m$ is the sum of the number of generators of all the ideals. 
	
	\item For each ideal $I_i$, construct a matrix $M(i)$ whose first row consists of the generators of the ideal and the second row consists of the indeterminates. 
	
	\item Add the ideals generated by $2 \times 2$ minors of these matrices to get an ideal $L.$ 
	
	\item To get the defining ideal, saturate $L$ with the product of $a_i$'s. 
\end{enumerate} 
{\bf Output}: The defining ideal of the Rees algebra $\R(I_1,\ldots,I_s).$ 	
\end{algorithm*}

The elements of the defining ideal are assigned $\mathbb{N}^{s+1}$ degree by the function, where the first $\mathbb{N}^{s}$ coordinates point to the component of $\mathcal{R}(I_1,\ldots,I_s)$ where the element lies and the last coordinate is the degree of the element. Observe that the defining ideal in the following example is generated in degrees $(1,0,5)$. The first two coordinates asserts that the generator lies in $I_1t_1$, whereas the last coordinate specifies that the element belongs to ${(I_1)}_5\subseteq R_5$.

\begin{example}
	{\small
		\setlength{\columnsep}{-15pt}
		\begin{multicols}{2}
			\begin{verbatim}
				i1 : R = QQ[w,x,y,z];
				i2 : I = ideal(x^2-y*w, x^3-z*w^2);
				i3 : J = ideal(w^2+x^2+y^2+z^2);
				i4 : M = multiReesIdeal {I,J};
				i5 : transpose gens M
				o5 = {-1, 0, -5} | (x3-w2z)X_0+(-x2+wy)X_1 |
				i6 : (first entries gens M)/degree 
				o6 : {{1, 0, 5}}   
			\end{verbatim}
		\end{multicols}
	}
\end{example} 

\subsection{Routine for the non-domain case} In this section we present an algorithm to find the defining ideal of the Rees algebra using the definition of Rees algebra.
This method does not have any requirements on the grade of the ideals or the domain property of the ring, but it seems to be comparably slower than the previous method. 

We can construct the Rees algebra of $I_i$ as the kernel of the map $\varphi_i:R[Y_{i1},\dots,Y_{in_i}]\rightarrow R[T_i]$ where $\varphi_i(Y_{ij})=f_{ij}T_i$ for $j=1, \ldots, n_i$. Notice that the idea $(\ker\varphi_i )R[\underline{Y}]\subseteq \ker \varphi$. Suppose that $\phi_i$ is the presentation matrix of $I_i$ for $i=1,\dots,s$. Then the symmetric algebra $\sym(I_i)$ has a presentation $R[Y_{i1},\dots,Y_{in_i}]/\mathcal{L}_i$ where $\mathcal{L}_i=I_1([Y_{i1},\dots,Y_{in_i}]\cdot \phi_i)$. Clearly, $\mathcal{L}_i \subseteq \ker \varphi_i\subset \ker\varphi$. So the map $\varphi_i$ factors through the symmetric algebra $\sym(I_i)$. Now $\sym(I_1)\otimes\cdots\otimes \sym(I_s)$ has the presentation $R[\underline{Y}]/(\mathcal{L}_1+\cdots+\mathcal{L}_s)$. Since each of $\mathcal{L}_i\subseteq\ker \varphi$, the map $\varphi$ also factors through $\sym(I_1)\otimes\cdots\otimes \sym(I_s)$. Thus to find the defining ideal of the multi Rees algebra $\mathcal{R}(I_1,\dots,I_s)$ it is enough to find the kernel of the surjective map $\sym(I_1)\otimes\cdots\otimes \sym(I_s)\rightarrow \mathcal{R}(I_1,\dots,I_s)$.

\vspace{0.15cm}
\noindent
\begin{algorithm*}[Version II: \texttt{multiReesIdeal}, no assumptions] 
Let $I_1,\ldots,I_s$ be ideals in the Noetherian ring $R$. \\
{\bf Input}: The list $W = \{I_1,\ldots,I_s\}.$
\begin{enumerate}[1.]
	\item For each $I_i$ compute the presentation $F_i'\xrightarrow{\phi_i} F_i \rightarrow I_i\rightarrow 0$, where $\phi_i$ is the presentation matrix of $I_i$ for $i=1,\dots,s$. 
	\item Now compute the source symmetric algebra  $\sym(F_1')\otimes\cdots\otimes\sym(F_s')$ and the target symmetric algebra $\sym(F_1)\otimes\cdots\otimes\sym(F_s)$ of the map $\phi_1 \otimes \cdots \otimes \phi_s$.
	\item Compute the map between the symmetric algebra of the source and target and return kernel of the above map.
\end{enumerate} 
{\bf Output}: The defining ideal of the Rees algebra $\R(I_1,\ldots,I_s).$ 
\end{algorithm*}

In the following example the ring $U$ is not a domain and hence the algorithm uses the above method. As expected, the computational time in the case a nonzerodivisor is given as an optional input is faster than the case where no optional input is given.

\begin{example}
	{\small 
		\begin{verbatim}
			i1 : R = QQ[w,x,y,z]/ideal(w*x, y*z);
			i2 : m = ideal vars R;
			i3 : time multiReesIdeal m
			-- used 0.290619 seconds
			o3 = ideal(y*X, z*X, z*X -x*X, y*X -x*X, w*X, z*X -w*X, y*X -w*X, x*X, X X, X X)
			              3    2    1    3    1    2    1    0    3    0    2    0  2 3  0 1
			i4 : time multiReesIdeal (m, w+x);
			-- used 0.0283374 seconds
		\end{verbatim}
	}  
\end{example}


\section{Computation of mixed multiplicities of multi-graded algebras}

In \cite{Verma-Katz-Mondal}, D. Katz, S. Mandal  and the last author, found a precise formula for the Hilbert polynomial of the quotient of a bi-graded algebra over an Artinian local ring. We generalize this result for quotient of a multi-graded algebra over an Artinian local ring. This improvement helps in the calculation of the mixed multiplicities of a set of ideals $I_0, I_1, \ldots, I_r$, where $I_0$ is primary to the maximal ideal and $I_0, I_1, \ldots, I_r$ are ideals in a local ring or in a standard graded algebra over a field. We end the section by giving an algorithm to compute the mixed multiplicity of a sequence of ideals $I_0,I_1,\ldots,I_r$ of a polynomial ring, where $I_0$ is an primary to the maximal ideal and $\hgt (I_j)>0$ for all $j.$ 

Let $S$ be an Artinian local ring and $A = S[X_1,\ldots,X_r]$ be an $\N^r$-graded ring over $S,$ where for $1 \leq i \leq r$, $X_i=\{X_i(0), \ldots, X_i(s_i)\}$ is a set of indeterminates of degree $e_i.$ Set $\underline{u}=(u_1, \ldots, u_r) \in \N^r$ and $|u|=u_1+\cdots+u_r$. Then $A = \bigoplus_{\underline{u}\in \N^r} A_{\underline{u}}$, where $A_{\underline{u}}$ is the $S$-module generated by monomials of the form $P_1\cdots P_r$, where $P_i$ is a monomial of degree $u_i$ in $X_i$. An element in $A_{\underline{u}}$ is called multi-homogeneous of degree $\underline{u}$. An ideal $I \subseteq A$ generated by multi-homogeneous elements is called a multi-homogeneous ideal. Then $R=A/I$ is an $\N^r$-graded algebra with $\underline{u}$-graded component $R_{\underline{u}}=A_{\underline{u}}/I_{\underline{u}}$. The Hilbert function of $R$ is defined as $H(\underline{u})=\lambda(R_{\underline{u}}),$ where $\lambda$ denotes the length as an $S$-module. Set $\underline{t}^{\underline{u}}=t_1^{u_1}\cdots t_r^{u_r}$. The Hilbert series of $R$ is given by $HS(R,\underline{t})=\sum_{\underline{u}\in \N^r} \lambda(R_{\underline{u}})\underline{t}^{\underline{u}}$. Then there exists a polynomial $N(t_1, \ldots, t_r) \in \Z[t_1, \ldots, t_r]$ so that 
\begin{equation}\label{hilSer}
HS(R,\underline{t})=N(t_1,\ldots, t_r)/\left((1-t_1)^{s_1+1} \cdots (1-t_r)^{s_r+1}\right). 
\end{equation}
Note that the rational function on the right hand side need not be reduced.

\begin{theorem}\label{thm1}
	Write the Hilbert polynomial of $R$ as
	\begin{equation}\label{eq11}
	P(\underline{u}, R) = \sum_{\alpha=\underline{0}}^{\underline{s}} c_{\alpha}\binom{u_1+\alpha_1}{\alpha_1} \cdots \binom{u_r+\alpha_r}{\alpha_r}.
	\end{equation}
	Then
	\[c_{\alpha}=\frac{(-1)^{|\underline{s}-\alpha|}}{(s_1-\alpha_1)! \cdots (s_r-\alpha_r)!} \cdot \frac{\partial^{|\underline{s}-\alpha|}N}{\partial t_1^{s_1-\alpha_1} \cdots \partial t_r^{s_r-\alpha_r}}\scalebox{2.5}{\ensuremath \mid}_{(t_1, \ldots, t_r)=\underline{1}}.\]
\end{theorem}

\begin{proof}
	We write 
	\[N^{(\underline{\alpha})} = \frac{\partial^{|\alpha|} N(t_1,\ldots, t_r)}{\partial t_1^{\alpha_1} \cdots \partial t_r^{\alpha_r}}\scalebox{2.5}{\ensuremath \mid}_{(t_1, \ldots, t_r)=\underline{1}}.\]Then \[HS(R,\underline{t})-\sum_{\underline{\alpha}=\underline{0}}^{\underline{s}}\frac{N^{(\underline{\alpha})}(-1)^{|\alpha|}}{\alpha_1!\cdots \alpha_r!(1-t_1)^{s_1+1-\alpha_1}\cdots(1-t_r)^{s_r+1-\alpha_r}}=\frac{D(t_1,\ldots, t_r)}{(1-t_1)^{s_1+1} \cdots (1-t_r)^{s_r+1}}\] where \[D(t_1,\ldots, t_r)=N(t_1,\ldots, t_r)-\sum_{\underline{\alpha}=\underline{0}}^{\underline{s}}\frac{N^{\underline{\alpha}}}{\alpha_1!\cdots \alpha_r!}(t_1-1)^{\alpha_1}\cdots(t_r-1)^{\alpha_r}.\] Thus $D(t_1,\ldots, t_r)$ is the remainder of the Taylor series of $N(t_1,\ldots, t_r)$ about the point $(1, \ldots, 1)\in \N^r$ having terms of degree $\geq s_i+1$ in $t_i-1$, for all $1 \leq i \leq r$. So $D(t_1,\ldots, t_r)$ is divisible by $(1-t_1)^{s_1+1} \cdots (1-t_r)^{s_r+1}$. Therefore, for all large $u_i$'s, $\lambda(R_{\underline{u}})$ is the coefficient of $\underline{t}^{\underline{u}}$ in the power series expansion of 	
	\begin{equation}\label{eq1}
	E(t_1,\ldots,t_r) = \sum_{\underline{\alpha}=\underline{0}}^{\underline{s}} \frac{N^{(\underline{\alpha})}(-1)^{|\underline{\alpha}|}}{\alpha_1!\cdots \alpha_r!(1-t_1)^{s_1+1-\alpha_1}\cdots(1-t_r)^{s_r+1-\alpha_r}}.
	\end{equation} 
	As the coefficient of $\underline{t}^{\underline{u}}$ in $E(t_1,\ldots, t_r)$ is given by a polynomial for all $\underline{u}$ so	
	\begin{equation}\label{eq2}
	E(t_1,\ldots, t_r)=\sum_{\underline{u} \in \N^r} P(\underline{u}; R)\underline{t}^{\underline{u}}.
	\end{equation}
	
	Here we are using the fact that two polynomials in $\Z[y_1, \ldots, y_r]$ coinciding at $\underline{u}$, for all $u_i$ large, are equal. Now expanding the rational function in \eqref{eq1}, we get
	\begin{equation}\label{eq3}
	\begin{split}
	E(t_1,\ldots, t_r)&=\sum_{\underline{\alpha}=\underline{0}}^{\underline{s}}\frac{N^{(\underline{\alpha})} (-1)^{|\underline{\alpha}|}}{\alpha_1!\cdots \alpha_r!}\sum_{\underline{u} \in \N^r}\binom{s_1-\alpha_1+u_1}{u_1}\cdots \binom{s_r-\alpha_r+u_r}{u_r}\underline{t}^{\underline{u}}\\
	&=\sum_{\underline{u} \in \N^r}\left(\sum_{\underline{\alpha}=\underline{0}}^{\underline{s}}\frac{N^{(\underline{\alpha})} (-1)^{|\underline{\alpha}|}}{\alpha_1!\cdots \alpha_r!}\right)\binom{s_1-\alpha_1+u_1}{u_1}\cdots \binom{s_r-\alpha_r+u_r}{u_r}\underline{t}^{\underline{u}}.
	\end{split}
	\end{equation}
	Comparing \eqref{eq3} with \eqref{eq2}, we get the result.
\end{proof}

Note that 
\[\binom{u_i+\alpha_i}{\alpha_i}= \frac{1}{\alpha_i!}u_i^{\alpha_i}+\text{ lower degree terms}.\] 
So if we write $P(\underline{u})$ as in \eqref{eq11}, then $c_{\alpha}=e_{\alpha}$ for all $\alpha \in \N^{r+1}$ with $|\alpha|=d-1$, see Theorem \ref{deg-mixedmul}.
Therefore, Theorem \ref{thm1} gives an expression for $e_{\alpha}.$

\begin{remark}
Let $I'_0, I'_1,\ldots,I'_r$ denote the images of ideals $I_0, I_1,\ldots,I_r$ in $A':=A/(0 : I^\infty),$ where $I = I_1 \cdots I_r.$ Following the notations in \cite{Trung-Verma}, we set $R:=R(I_0 |I_1,\ldots, I_r)=\R(I_0, \ldots, I_r)/(I_0)$ and put $R':= R(I'_0 |I'_1,\ldots, I'_r).$ Notice $R=\gr_A(I_0,I_1,\ldots, I_r; I_0)$, defined in \eqref{gr}. Then for $\underline{u}$ large, $P_R(\underline{u}) = P_{R'}(\underline{u})$ (see \cite[Theorem 1.2]{Trung-Verma} for details). Therefore, in case $\grade I_i=0$ for some $i$, the user needs to work in the quotient ring $A'$ and input the images of the ideals in the quotient ring.
\end{remark}

\noindent
\begin{algorithm*}
The algorithm for the function {\rm \texttt{mixedMultiplicity}} uses the above ideas to calculate the mixed multiplicity. Let $I_0,I_1,\ldots,I_r$ be a set of ideals of a Noetherian ring $R$ of dimension $d \geq 1$, where $I_0$ is primary to the maximal ideal and $\grade(I_i)>0$ for all $i$; $\underline{a}=(a_0,a_1,\ldots,a_r) \in \N^{r+1}$ with $|\underline{a}|=d-1.$ \\
{\bf Input}: The sequence $W=((I_0,I_1,\ldots,I_r),(a_0,a_1,\ldots,a_r))$.
\begin{enumerate}[1.]
	\item Compute the defining ideal of the multi-Rees algebra using the function {\rm \texttt{multiReesIdeal}} and use it to find the Hilbert series of $R(I_0 \mid I_1,\ldots,I_r).$
	
	\item Extract the powers of $(1-T_i)$ in the denominator of the Hilbert series.
	
	\item Calculate $e_{\underline{a}}$ using the formula given in Theorem \ref{thm1}.
\end{enumerate}
{\bf Output}: The mixed multiplicity $e_{\underline{a}}(I_0 \mid I_1,\ldots,I_r)$.
\end{algorithm*} 

\begin{example}
	{\small
		\begin{multicols}{2}
			\begin{verbatim}
				i1 : R = QQ[w,x,y,z];
				i2 : I = ideal(x^2-y*w, x^3-z*w^2);
				i3 : m = ideal vars R;
				i4 : mixedMultiplicity((m,I),(3,0))
				o4 = 1
				i5 : mixedMultiplicity((m,I),(2,1))
				o5 = 2
			\end{verbatim}
		\end{multicols}
	}
\end{example}

When some ideal has grade zero, the following example explains how to compute the mixed
multiplicity.

\begin{example}
	Let $R=\mathbb{Q}[w,x,y,z]/(wx,yz)$, $\m=(w,x,y,z)$, and $I=(w,y).$ Notice that $\grade I=0$, since $I \in \Ass R$.
	{ \small
		\begin{multicols}{2}
			\begin{verbatim}
			i1 : R = QQ[w,x,y,z]/ideal(w*x, y*z);
			i2 : I = ideal(w,y);
			i3 : m = ideal vars S;
			i4 : L = saturate(sub(ideal 0, R), I);
			i5 : T = S/L;
			i6 : J = substitute(I, T);
			i7 : n = substitute(m, T); 
			i8 : dim T
			o8 = 2
			i9 : mixedMultiplicity ((n,J),(1,0))
			o9 = 3
			\end{verbatim}
		\end{multicols}
	}	
\end{example}

To calculate mixed multiplicities, the function \texttt{mixedMultiplicity} computes the Hilbert polynomial of the graded ring $\bigoplus I_0^{u_0}I_1^{u_1} \cdots I_r^{u_r}/I_0^{u_0+1}I_1^{u_1} \cdots I_r^{u_r}$ . In particular, if $I_1,\ldots,I_r$ are also $\m$-primary ideals, then $e_{(a_0,a_1,\ldots,a_r)}(I_0 \mid I_1,\ldots,I_r) = e(I_0^{[a_0+1]}, I_1^{[a_1]}, \ldots, I_r^{[a_r]})$ (see \cite[Definition 17.4.3]{Huneke-Swanson}). Therefore, to compute the $(a_0+1, a_1,\ldots, a_r)$-th
mixed multiplicity of $I_0,I_1,\ldots,I_r$, one needs to enter the sequence $(a_0,a_1,\ldots,a_r)$ in the function. The same is illustrated in the following example.

\begin{example}
	Let $R=\mathbb{Q}[w,x,y,z]$ and $\m$ be the maximal homogeneous ideal of $R.$ Let $I = (x^2-yw, x^3-zw^2)$ and $J = \m^4 + I.$ We calculate the mixed multiplicities $e(\m) = e(\m^{[4]}, J^{[0]})$ and $e(\m^{[3]},J^{[1]}).$	
	{\small\begin{multicols}{2}
			\begin{verbatim}
				i1 : R = QQ[w,x,y,z];
				i2 : m = ideal vars R;
				i3 : I = ideal(x^2-y*w, x^3-z*w^2);
				i4 : J = m^4 + I;
				i5 : mixedMultiplicity ((m,J),(3,0))
				o5 = 1
				i6 : mixedMultiplicity ((m,J),(2,1))
				o6 = 2
			\end{verbatim}
	\end{multicols}}
\end{example}

\section{Writing the function \texttt{mixedMultiplicity} in Macaulay2}
In this section, we illustrate a detailed overview of the function \texttt{mixedMultiplicity} based on the algorithm discussed in the previous section. 
One can refer to the package \href{https://faculty.math.illinois.edu/Macaulay2/doc/Macaulay2-1.19.1/share/doc/Macaulay2/MixedMultiplicity/html/index.html}{\texttt{MixedMultiplicity}} for the Macaulay2 code of this function. 

We start with two sequences \texttt{W1}=$(I_0, I_1, \ldots, I_r)$ and \texttt{W2}=$(a_0, a_1,\ldots,a_r)$. Then the command \texttt{W1\#i} returns the $i$-th element of \texttt{W1}, that is, $I_i$. All ideals are expected to be in the same ring. To capture the ambient of these ideals we use \texttt{ring W1\#0}. We now check that the ambient ring has dimension at least one, the two sequences \texttt{W1}, \texttt{W2} have same length, \texttt{W1} consists of ideals of positive grade with same ambient ring, and \texttt{W2} is a sequence of natural numbers. Now we set `$n$' to be the cardinality of \texttt{W1}, that is, $n=r+1$. The defining ideal of the multi-Rees algebra $\R(I_0, \ldots, I_r)$ is computed using the function \texttt{multiReesIdeal} and named as \texttt{L1}. For $i=1, \ldots, n$, \texttt{V(i)} gives us the unit vector $(0,..,0,1,0,..,0) \in \N^n$, where $1$ is at the $i^{th}$ place. Next, \texttt{Q} is a list of \texttt{V(i)}'s where \texttt{V(i)} occurs as many times as the number of generators of \texttt{W1\#i} (i.e., $I_i$) for $i=0, \ldots,r$. We use \texttt{compress gens} to remove zero entries, if any, from the input generating set of ideals. As we use \texttt{compress gens} in the function \texttt{multiReesIdeal} too so it is compatible. Since \texttt{gens} gives us the generators of an ideal as a matrix, we 
could also use \texttt{\#(first entries compress gens W1\#i)}. The command \texttt{first entries} extracts the first row of a matrix in the form of a list. Other possible commands are \texttt{numgens trim W1\#i} and \texttt{rank source mingens W1\#i}. 

Then we define a ring \texttt{S}, which is the same ring as the ambient ring of \texttt{L1}, but is now multigraded with degrees given by \texttt{Q}.
By \texttt{sub(W1\#0, S)}, we view the ideal \texttt{W1\#0} as an ideal in \texttt{S} and call it \texttt{p}. Next, we view \texttt{L1} as an ideal in \texttt{S/p} and call it \texttt{L2}. Recall that \texttt{W1\#0}$=I_0$ is primary to the maximal ideal in its ambient ring. \texttt{hilbertSeries} computes the Hilbert series of \texttt{L2} and reduce it using the command \texttt{reduceHilbert}. Thus we get an expression \texttt{H} as in \eqref{hilSer}. To use Theorem \ref{thm1}, we extract the numerator \texttt{B0} and denominator \texttt{B1} of \texttt{H}. Factoring out \texttt{B1}, we capture the powers. For instance, if \texttt{B1} $ = T_0T_1^2-2T_0T_1-T_1^2+T_0+2T_1-1$, then \texttt{facB1 = factor B1} $ = (T_1-1)^2(T_0-1)$, and \texttt{facB1\#0} $ = (T_1-1)^2.$ The command \texttt{facB1\#0} executes because \texttt{factor} outputs an expression of the class \texttt{Product}, which is a type of a basic list. Since \texttt{Power} is also a type of a basic list, we get \texttt{facB1\#0\#1}= $2.$ We now extract $a_i$'s with the command \texttt{W2\#i} and verify that $a_i \leq \mbox{ power of } (1-T_i).$ At this stage the function has the value of all the variables in the expression given in Theorem \ref{thm1}. So we just compute the value.

\section{Mixed volume of lattice polytopes}

Given a collection of lattice polytopes in $\mathbb{R}^n$, Trung and the last author proved that their mixed volume is equal to a mixed multiplicity of a set of homogeneous ideals, see Proposition \ref{vol-mul}. We use this result to construct an algorithm which calculates the mixed volume of a collection of lattice polytopes. We also give an algorithm which outputs the homogeneous ideal corresponding to the vertices of a lattice polytope.

Let $Q$ be a lattice polytope in $\mathbb{R}^n$ with the set of vertices $\{p_1,\ldots,p_r\} \subseteq \N^n$. We first compute the corresponding homogeneous ideal $I$ in the polynomial ring $R=k[x_1,\ldots,x_{n+1}]$ such that $Q$ is the convex hull of the lattice points of the dehomogenization of a monomial generating set of $I$ in $k[x_1,\ldots,x_n]$. For this purpose, we write a function \texttt{homIdealPolytope} which requires as an input the list of vertices of $Q$.

\begin{algorithm*}  
Suppose that $p_1,\ldots,p_r \subseteq \N^n$ are vertices of a polytope $Q$.
\\
{\bf Input}: The list $W=\{p_1,p_2,\ldots, p_r\}$.
\begin{enumerate}[1.]
	\item Define a polynomial ring $R$ in $(n+1)$-variables over $\mathbb{Q}$, where the vertices are in $\mathbb{R}^n$. 
	
	\item For each vertex $p_i$, associate a monomial in the first $n$-variables, where the $j$-th component of $p_i$ is the power of the $j$-th variable and add them up.
	
	\item Add the monomials and homogenize it with respect to the last variable. 
	
	\item Define an ideal generated by the terms of the homogeneous element.
\end{enumerate} 
{\bf Output}:  The corresponding homogeneous ideal in $\mathbb{Q}[x_1,\ldots, x_{n+1}]$.
\end{algorithm*} 

We now write a function \texttt{mMixedVolume} to calculate the mixed volume of a collection of $n$ lattice polytopes in $\mathbb{R}^n.$ Let $Q_1,\ldots,Q_n$ be an arbitrary collection of lattice polytopes in $\mathbb{R}^n$. Let $I_i$ be the homogeneous ideal of $R=k[x_1,\ldots,x_{n+1}]$ associated to $Q_i$ for $i=1,\ldots, n$.
The function \texttt{mMixedVolume} takes the list $\{I_1,\ldots, I_n\}$ as an input and produces the mixed volume of $Q_1, \ldots, Q_n$ as an output. The function can also take the list of lists of vertices of the polytope as an input to compute their mixed volume. Since calculating the mixed volume is same as calculating a mixed multiplicity, the algorithm of the function \texttt{mMixedVolume} is similar to that of \texttt{mixedMultiplicity}.

\begin{example}
	Let $Q_1, Q_2, Q_3$ be the same tetrahedron with vertices $(1,1,0),(2,1,0),(1,3,0)$ and $(1,1,3)$. Using the formula, the mixed volume can be calculated as follows: 
	\begin{align*} 
	MV_3(Q_1,Q_2,Q_3)
	=& V_3(Q_1+Q_2+Q_3) - V_3(Q_1+Q_2) - V_3(Q_2+Q_3) - V_3(Q_3+Q_1) \\
	&\qquad + V_3(Q_1) + V_3(Q_2) + V_3(Q_3),
	\end{align*}
	where $V_3$ denotes the $3$-dimensional Euclidean volume. One can check that $V_3(Q_i)=1$, for all $i=1,2,3$; $V_3(Q_i+Q_j)=8$, for all $i \neq j$ and $V_3(Q_1+Q_2+Q_3)=27$. Hence $MV_3(Q_1,Q_2,Q_3)=27-(3 \times 8)+(3 \times 1)=6.$ The following Macaulay2 working verifies this calculation.
	{\small \begin{verbatim}
		i1 : A = {(1,1,0),(2,1,0),(1,3,0),(1,1,3)};
		i2 : mMixedVolume{A,A,A}
		o2 = 6
	\end{verbatim}
}
\end{example}
%
%
%

\begin{example}
	Consider a generic vertical parabola $f_1$ and a generic horizontal parabola $f_2$ given by 
	\begin{align*}
		f_1&=c_{1,(0,0)}+c_{1,(1,0)}x+c_{1,(2,0)}x^2+c_{1,(0,1)}y,\\
		f_2&=c_{2,(0,0)}+c_{2,(1,0)}x+c_{2,(0,1)}y+c_{2,(0,2)}y^2.
	\end{align*}
	We compute the mixed volume of the Newton polytopes associated to the supports $Q_1, Q_2$ of $f_1$, and $f_2$ respectively.
	{ \small
		\begin{multicols}{2}
			\begin{verbatim}
				i1 : Q1 = {(0,0),(1,0),(2,0),(0,1)};
				i2 : Q2 = {(0,0),(1,0),(0,1),(0,2)};
				i3 : mMixedVolume {Q1, Q2}
				o3 = 4
			\end{verbatim}
		\end{multicols}	}
	This computation verifies \cite[Example 1.2]{Averkov2021} and also gives a bound for the number of common zeros (in the complex torus) of $f_1$ and $f_2$ (see Bernstein's theorem in subsection \ref{MV}). 
\end{example}


\begin{example}
	We compute and verify the mixed volumes of the two pairs of convex polytopes, which appeared in \cite[Example 1.1]{Chen} and \cite[Section 2]{Chen} (take $\lambda=1$), respectively. 
		\begin{verbatim}
			i1 : Q1 = {(1, 1),(3, 0),(4, 0),(4, 1),(3, 3),(1, 4),(0, 4),(0, 3)};
		\end{verbatim}
	\begin{minipage}{.725 \columnwidth}
	\begin{verbatim}
			i2 : Q2 = {(0, 1),(0, 0),(3, 0),(4, 1),(4, 4),(3, 4)};
			i3 : mMixedVolume {Q1, Q2}
			o3 = 32
			i4 : P1 = {(0, 0),(0, 2),(2, 0),(2, 2)};
			i5 : P2 = {(0, 0),(1, 2),(2, 1)};
			i6 : mMixedVolume {P1, P2}
			o6 = 8
	\end{verbatim}
	\end{minipage}
	\begin{minipage}{.225 \columnwidth}
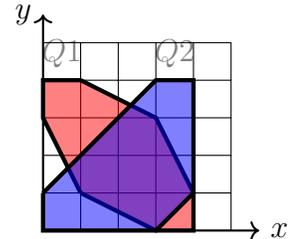
		\begin{center}
			\begin{tikzpicture}[scale=0.5]
				\draw[->, thick] (0,0)--(5.75,0) node[right]{$x$};
				\draw[->, thick] (0,0)--(0,5.75) node[left]{$y$};
				\draw (0,0) grid (5,5);
				\draw[-,ultra thick,fill=red,fill opacity=0.5](1,1)-- (3,0)-- (4,0)-- (4,1)-- (3,3)-- (1,4)to node[above]{$Q1$} (0,4)-- (0,3)-- (1,1);
				\draw[-,ultra thick,fill=blue,fill opacity=0.5](0,1)-- (0,0)-- (3,0)-- (4,1)-- (4,4)to node[above]{$Q2$} (3,4)-- (0,1);
			\end{tikzpicture}
		\end{center}
	\end{minipage}
\end{example}

\section{Sectional Milnor number}

In this section, we give an algorithm to compute the sectional Milnor numbers. We use Teissier's observation of identifying the sectional Milnor numbers with mixed multiplicities to achieve this task. Teissier (\cite{teissier1973}) conjectured that invariance of Milnor number implies invariance of the sectional Milnor numbers. The conjecture was disproved by Jo\"el Brian\c{c}on and Jean-Paul Speder. We verify their example using our algorithm and also by explicitly calculating the mixed multiplicities.

Let $R=\mathbb{C}[x_1,\ldots,x_n]$ be a polynomial ring in $n$ variables, $\mathfrak{m}$ be the maximal graded ideal and $f \in R$ be any polynomial with an isolated singularity at the origin. Using Theorem \ref{thm1}, one can now calculate the mixed multiplicities of $\m$ and $J(f)$. We use the ideas in the previous section to write a function \texttt{secMilnorNumbers} for computing the first $n-1$ sectional Milnor numbers. 
\noindent
With a polynomial $f$ given as an input, the algorithm calculates the Jacobian ideal of $f$ and then using the function \texttt{multiReesIdeal}, it finds the defining ideal of $\R(\m, J(f))$. This helps to compute the Hilbert series of the special fiber $\f(\m,J(f))= \R(\m, J(f)) \otimes_R R/\mathfrak{m}$. Using the formula given in Theorem \ref{thm1}, it then calculates the mixed multiplicities. Note that the $n^{th}$-sectional Milnor number is the Milnor number of the hypersurface $f$ at the origin. So under the extra assumption that the ideal $J(f)$ is $\m$-primary, we have $\mu^{(n)}(X, 0)=\dim_{\mathbb{C}} R/J(f)$. Together, the function \texttt{secMilnorNumbers} outputs $\left(\mu^{(0)}(X, 0), \mu^{(1)}(X, 0),\ldots, \mu^{(n)}(X, 0)\right)$.

\begin{example}
Let $R=\mathbb{Q}[x,y,z]$ and $f=x^4+y^4+z^4$ be a polynomial in the ring. 
	{\small
			\begin{verbatim}
			i1 : R = QQ[x,y,z];
			i2 : f = x^4+y^4+z^4;
			i3 : secMilnorNumbers f
			o3 = HashTable{0 => 1, 1 => 3, 2 => 9, 3 => 27}
			\end{verbatim}
	}
\end{example}

\begin{remarks}
	
\noindent	
\begin{enumerate}
\item We set $I:=(f, J(f))$. If $f$ has an isolated singularity at the origin, then by definition $IR_{\m}$ is an $\m R_\m$-primary ideal. In view of \cite[Theorem 7.1.5, Proposition 1.6.2]{Huneke-Swanson}, $f/1\in \overline{J(f)R_\m}$ and hence $J(f)R_\m \subset I R_\m$ is a reduction. Thus 
\[IR_\m \subset \overline{IR_\m}=\overline{J(f)R_\m} \subseteq \sqrt{J(f)R_\m} \subset \sqrt{IR_\m},\]
see \cite[Remark 1.1.3 (3)]{Huneke-Swanson}. So $\sqrt{J(f)R_\m} = \sqrt{IR_\m}=\m R_\m$. Therefore, $J(f)R_\m$ is $\m R_\m$-primary. The example in $(2)$ says that $J(f)$ need not be $\m$-primary even when $I$ is $\m$-primary. However, if $f$ is homogeneous, then by the \emph{classical Euler formula} $(\deg f)\cdot f=$ $\sum_{i=1}^n x_i\cdot \partial f/\partial x_i \in J(f)$. Hence $J(f)=I$. Consequently, $J(f)$ is $\m$-primary if $I$ is so.	
\\[1pt]

\item If $J(f)$ is not $\m$-primary, then the Milnor number is $\dim_{\mathbb{C}} \frac{R_\m}{J(f)R_\m}$, the multiplicity of $R/J(f)$ at $\m$. It can be computed using Macaulay2 by means of the $R$-module isomorphism 
\begin{equation}\label{iso_sat}
\frac{R_\m}{J(f)R_\m} \overset{\overline{\varphi}}{\cong} \frac{R}{\varphi^{-1}\big(J(f)R_\m\big)}= \frac{R}{(J(f):(J(f):\m^\infty))}
\end{equation}
induced from the natural inclusion $\varphi: R \to R_\m$, see \cite[Page 61]{eisenbud2001comp}. Since $\varphi^{-1}\big(J(f)R_\m\big)R_\m=J(f)R_\m$ by \cite[Proposition 1.17]{Atiyah-MacDonald}, $\overline{\varphi}_\m$ is an isomorphism. So from \cite[Proposition 3.9]{Atiyah-MacDonald} it follows that $\overline{\varphi}$ is an isomorphism. It is now enough to show 
\[\varphi^{-1}\left(J(f)R_\m\right)=(J(f):(J(f):\m^\infty)).\]
Suppose that $J(f)=\left(\cap_{i=1}^r Q_i\right) \cap Q$ is a minimal primary decomposition, where $\sqrt{Q}=\m$ and $\sqrt{Q_i}=P_i \neq \m$ for $i=1, \ldots, r$. Since $J(f)R_\m$ is $\m R_\m$-primary, $\m$ is a minimal prime of $J(f)$ and $P_i\nsubseteq \m$ for $i=1, \ldots, r$. By \cite[Theorem 4.10]{Atiyah-MacDonald}, the $\m$-primary component of $J(f)$ is unique.
We now show that 
\[\varphi^{-1}\left(J(f)R_\m\right)=Q=(J(f):(J(f):\m^\infty)).\]
By \cite[Proposition 4.9]{Atiyah-MacDonald}, $J(f)R_\m=QR_\m$ and $\varphi^{-1}(J(f)R_\m)=Q$. On the other hand, $\widetilde{J(f)}:=(J(f):\m^\infty)=\cap_{i=1}^r Q_i$, the intersection of all non $\m$-primary components of $J(f)$. Due to \cite[Proposition 1.11]{Atiyah-MacDonald}, we can pick $x \in \left(\cap_{i=1}^r Q_i\right) \backslash \m$. Then 
\[(J(f): \widetilde{J(f)}) \subseteq (J(f):x)=\left(\cap_{i=1}^r(Q_i:x)\right) \cap (Q:x)=Q,\] 
by \cite[Lemma 4.4]{Atiyah-MacDonald}. Besides, as $\widetilde{J(f)}, Q$ are co-maximal so $Q \cdot \widetilde{J(f)}=Q \cap \widetilde{J(f)}=J(f)$. Therefore, $Q \subseteq (J(f): \widetilde{J(f)})$ and hence $(J(f): \widetilde{J(f)})=Q$. Note that the ideal $J(f)$ is $\m$-primary if and only if $J(f):\m^\infty=R$.
\\[1pt]

\item If $J(f)$ is not $\m$-primary, then one can use the isomorphism in \eqref{iso_sat} to compute the Milnor number of $f$. However, computing `saturation' is time-consuming. Besides, the Macaulay2 commands for working in a local ring are not fully developed. So we decided not to use either of these in the function \texttt{secMilnorNumbers} and restricted ourselves to the case when $J(f)$ is $\m$-primary. This assumption forces $\mathrm{Sing}(f)=\{\mbox{origin}\}$, as the ideal $(f, J(f))$ is also $\m$-primary. Nevertheless, users can use the succeeding commands to compute the Milnor number of $f$ directly, whether $J(f)$ is $\m$-primary or not. We take $f=x^2+y^2+z^2+xyz \in \Q[x,y,z]$. By \cite[Example 2.4]{liu2018}, the Milnor number of $f$ is $1$. Notice $(f, J(f))$ is an $\m$-primary ideal. 
	{\small
	\begin{multicols}{2}
	\begin{verbatim}
	i1 : R = QQ[x,y,z];
	i2 : f = x^2+y^2+z^2+x*y*z;
	i3 : J = ideal jacobian f;
	i4 : m= ideal vars R;
	i5 : J1= saturate(J, m);
	i6 : J1 == sub(ideal 1, R)
	o6 : false
	i7 : I= ideal f + J;
	i8 : saturate(I, m) == sub(ideal 1, R)
	o8 : true
	i9 : J2 = J : J1;
	i10 : degree J2
	o10 : 1
	-- the Milnor number of f
	\end{verbatim}
	\end{multicols}
}
\end{enumerate}
\end{remarks}

Let $h$ be a non-constant homogeneous polynomial in $\mathbb{C}[z_0,\ldots,z_n]$ and let $\Delta_h \subseteq \mathbb{R}^n$ be the convex hull of exponents of dehomogenized monomials  in $\mathbb{C}[z_1,\ldots,z_n]$ appearing in one of the partial derivatives of $h.$ Note that such $h$ need not have an isolated singularity at the origin. We set $\m=(z_0,\ldots,z_n)$. In \cite{June12}, June Huh extended the notion of sectional Milnor numbers and defined a sequence $\{\mu^i(h)\}_{i=0}^n$ by setting $\mu^i(h)=e_i(\m, J(h))$ for $i=0, \ldots, n$. We want to remark that once can compute $\mu^i(h)$ using the command \texttt{mixedMultiplicity}. 
Huh compared $\mu^i(h)$ with the mixed volume of standard $n$-dimensional simplex $\Delta$ and $\Delta_h$ in $\mathbb{R}^n.$ 
\begin{theorem}[{\cite[Theorem 15]{June12}}] \label{June}
Suppose that $h$ is a non-constant homogeneous polynomial in $\mathbb{C}[z_0,\ldots,z_n]$. For $i = 0,\ldots, n$, we have
	\[ \mu^{(i)}(h) \leq MV_n(\underbrace{\Delta,\ldots,\Delta}_{n-i},\underbrace{\Delta_h,\ldots,\Delta_h}_{i}).\]
\end{theorem}
We use our algorithms to verify the above result through examples.

\begin{example}
Let $h=z_0^3+z_1^3+z_2^3+z_0z_1z_2\in \mathbb{Q}[z_0,z_1,z_2]$.	Then the Jacobian ideal $J(h)=(h_{z_0},h_{z_1},h_{z_2})$ is $(3z_0^2+z_1z_2,3z_1^2+z_0z_2,3z_2^2+z_0z_1)$. We set $\m=(z_0, z_1, z_3)$. Using Macaulay2, we verify that $\hgt J(h)=3$. This ensures that $J(h)$ is an $\m$-primary ideal, since $J(h)$ is homogeneous. Consequently, $h$ has an isolated singularity at the origin. Notice that the vertices of $\Delta_h$ are $\{(0,0),(1,1),(2,0),(0,1),(1,0),(0,2)\}$. The following session in Macaulay2 supports Theorem \ref{June}. In fact, we get equality for all $i$. 
{\small\begin{multicols}{2}
		\begin{verbatim}
			i1 : R = QQ[z_0..z_2];
			i2 : h = z_0^3+z_1^3+z_2^3+z_0*z_1*z_2;
			i3 : J = ideal jacobian ideal h;
			i4 : codim J
			o4 : 3
			i5 : I = homIdealPolytope
			    {(0,0),(1,1),(2,0),(0,1),(1,0),(0,2)};
			i6 : W = ring I;
			i7 : n = ideal vars W;
			i8 : mMixedVolume {n,n}
			o8 : 1
			i9 : mMixedVolume {n,I}
			o9 : 2
			i10 : mMixedVolume {I,I}
			o10 : 4
			i11 : secMilnorNumbers(h)
			o11 : HashTable{0=>1, 1=>2, 2=>4, 3=>8}
		\end{verbatim}
\end{multicols}}

In \cite[Example 17]{June12}, J. Huh produced an example of a polynomial $h$ which gives a strict inequality for some $i$. Although, this $h$ does not have isolated singularity at the origin. The following Macaulay2 session verifies the example. 
{\small\begin{multicols}{2}
		\begin{verbatim}
			i1 : R = QQ[z_0..z_2];
			i2 : h = z_1*(z_0*z_1 - z_2^2);
			i3 : J = ideal jacobian ideal h;
			i4 : codim J
			o4 : 2
			i5 : m = ideal vars R;		
			i6 : mixedMultiplicity ((m,J),(2,0))
			o6 : 1
			i7 : mixedMultiplicity ((m,J),(1,1))
			o7 : 2
			i8 : mixedMultiplicity ((m,J),(0,2))
			o8 : 1
			i9 : I = homIdealPolytope 
			         {(2,0),(1,0),(0,2),(1,1)};
			i10 : W = ring I;
			i11 : n = ideal vars W;
			i12 : mMixedVolume {n,n}
			o12 : 1
			i13 : mMixedVolume {n,I}
			o13 : 2
			i14 : mMixedVolume {I,I}
			o14 : 2
		\end{verbatim}
\end{multicols}}
Note that for $i=2$, we are getting the inequality $1<2$ as it was claimed by Huh. One can also consider the homogeneous polynomial $h= z_0^2z_1+z_0^2z_2+z_1^2z_2+z_1z_2^2 \in \mathbb{Q}[z_0,z_1,z_2]$. Running a similar session as above, we get that $\mu^{(2)}(h) =2$ whereas $MV_2(\Delta_h,\Delta_h)=4.$
\end{example}

\subsection{Verifying example of Jo\"el Brian\c{c}on and Jean-Paul Speder}

We first recall Teissier's conjecture which claims that the invariance of Milnor number implies invariance of sectional Milnor numbers.

\noindent
{\bf Teissier’s Conjecture.} \cite{teissier1973} If $(X,x)$ and $(Y,y)$ have same topological type, then
\[\mu^*(X, x)=\mu^*(Y, y).\]
In \cite{Briancon-Speder}, Jo{\"e}l Brian\c{c}on and Jean-Paul Speder disproved the conjecture by giving a counter-example. They considered the family of hypersurfaces ${\bf X}_{t} \in \C^3$ defined by 
\[F_t(x,y,z)=z^5 + ty^6z + xy^7 + x^{15}=0.\] 
This family ${\bf X}_t$ provides a counter example to Teissier's conjecture. We verify the example by using Teissier's observation of identifying the sectional Milnor numbers with mixed multiplicities of ideals.

Consider the ideals $\m=(x,y,z)$ and $J(F_t)=(\partial F_t/\partial x, \partial F_t/\partial y, \partial F_t/\partial z)$ in $\mathbb{C}[x,y,z]$, where 
\[\frac{\partial F_t}{\partial x}= y^7+15x^{14},\quad \frac{\partial F_t}{\partial y}= 6ty^5z+7xy^6 \quad\text{ and }\quad \frac{\partial F_t}{\partial z}=5z^4+ty^6.\] 
We show that while $e(J(F_t))$ in independent of $t$, $e_2(\m \mid J(F_t))$ depends on $t$. Recall that 
\[e_3(\m \mid J(F_t))=e(J(F_t))=\ell\left(\frac{\C[x,y,z]}{J(F_t)}\right).\] 
Let $t=0.$ Since $J(F_0)$ is generated by a system of parameters,
\begin{align*}
e(J(F_0))=e(y^7+15x^{14},xy^6,z^4)
&=4e(y^7+15x^{14},xy^6,z)\\
&=4e(y^7+15x^{14},x,z)+4e(y^7+15x^{14},y^6,z)\\
&=28+24e(x^{14},y,z)=28+336=364.
\end{align*}
Here, we use a fact that in a Cohen-Macaulay ring, if $(a,b,c)$ and $(a,b,d)$ are ideals generated by system of parameters such that $c$ and $d$ are co-prime and $a,b,cd$ is a system of parameter, then $e(a,b,cd) = e(a,b,c)+e(a,b,d).$
Now let $t \neq 0$. Using the above observation and \cite[Theorem 14.11]{matsumuraCRT}, we get
\begin{align*}
e(J(F_t))&=e(y^7+15x^{14},y^5(6tz+7xy),5z^4+ty^6)\\
&=5e(y^7+15x^{14},y,5z^4+ty^6)+e(y^7+15x^{14},6tz+7xy,5z^4+ty^6)\\
&=5e(x^{14},y,z^4)+e(y^7+15x^{14},5x^4y^4\alpha^4+ty^6) \quad (\text{by putting $z=xy\alpha$ where $\alpha=-\frac{7}{6t}$})\\
&=280+e(y^7+15x^{14},y^4(5x^4\alpha^4+ty^2))\\
&=280+4e(y^7+15x^{14},y)+e(y^7+15x^{14},5(\alpha x)^4+ty^2)\\
&=280+56+e(y^7+15x^{14},\sqrt{5}\alpha^2x^2 + i\sqrt{t}y)+e(y^7+15x^{14},\sqrt{5}\alpha^2x^2 - i\sqrt{t}y)\\
&=336+14+14=364.
\end{align*}
Hence $e(J(F_t))=364$ and is independent of $t$. 
We now calculate $e_2(\m \mid J(F_t))$. 

In order to calculate $e_2(\m \mid J(F_t)),$ we find a joint reduction of $(\m, J(F_t),J(F_t)).$ Let $t=0.$ Consider the set of elements $\{ x,\partial F_0/\partial x,\partial F_0/\partial z \}.$ Using Macaulay2, we could check that 
\[ \m^{10}J^2 = (x)\m^9J^2 + (\partial F_0/\partial x,\partial F_0/\partial z)\m^{10}J \] 
in $\mathbb{Q}[x,y,z].$ This implies $\{ x,\partial F_0/\partial x,\partial F_0/\partial z \}$ is a joint reduction of $(\m, J(F_0),J(F_0))$ and hence 
\[ e_2(\m \mid J(F_0)) = e(x,\partial F_0/\partial x,\partial F_0/\partial z) = e(x,y^7+15x^{14},z^4) = e(x,y^7,z^4)=28.\] 
If $t \neq 0,$ we consider the set of elements $\{x,\partial F_t/\partial y,\partial F_t/\partial z\}.$ Using Macaulay2, we could check that 
\[ \m^{11}J^2=(x) \m^{10}J^2+(\partial F_t/\partial y, \partial F_t/\partial z)\m^{11}J \]
in the ring $\mathbb{Q}(t)[x,y,z].$ So $\{x,\partial F_t/\partial y,\partial F_t/\partial z\}$ is a joint reduction of $(\m, J(F_t), J(F_t))$ and hence
\begin{align*}
e_2(\m \mid J(F_t)) = e(x,\partial F_t/\partial y,\partial F_t/\partial z) 
&= e(x, 6ty^5z+7xy^6, 5z^4+ty^6) \\
&= e(x, y^5z, 5z^4+ty^6) \\
&= e(x,y^5,z^4) + e(x,z,y^6) = 20+6 = 26 .
\end{align*}

This proves that the mixed multiplicities are dependent on $t$, verifying the example given by Brian\c{c}on and Speder. The following displays the working in Macaulay2.
{\small
	\setlength{\columnsep}{-20pt}
	\begin{multicols}{2}
		\begin{verbatim}
		i1 : k = frac(QQ[t]);
		i2 : R = k[x,y,z];
		i3 : f = z^5 + t*y^6*z + x*y^7 + x^15;
		i4 : secMilnorNumbers (f)
		o4 = HashTable{0=> 1, 1=> 4, 2=> 26, 3=> 364}
		i5 : g = z^5 + x*y^7 + x^15;
		i6 : secMilnorNumbers (g)
		o6 = HashTable{0=> 1, 1=> 4, 2=> 28, 3=> 364}
		\end{verbatim}
	\end{multicols}
}

\bibliographystyle{plain}

\end{document}